\documentclass[12pt,reqno]{amsart}

\usepackage{amsmath,amsthm,amssymb,color}
\usepackage[a4paper, left=38mm, right=38mm, top=45mm, bottom=50mm, headsep=10mm, footskip=10mm]{geometry}

\newtheorem{theorem}{Theorem}

\newtheorem{lemma}[theorem]{Lemma}

\def\v{v}
\def\n{n}
\def\x{x}
\def\d{\,{\rm d}}

\def\phipsi{\phi_{\bar\psi}}
\def\vd{\v \cdot \nabla}

\def\eps{\varepsilon}
\def\ueps{u_\eps}
\def\gameps{\gamma_\eps}
\def\sigeps{\sigma_\eps}
\def\feps{f_\eps}
\def\geps{g_\eps}

\def\R{\mathcal{R}}
\def\V{\mathcal{V}}

\def\D{{\R \times \V}}
\def\dR{{\partial \R}}
\def\dD{{\dR \times \V}}
\def\G{\Gamma}

\def\RR{\mathbb{R}}
\def\VV{\mathbb{V}}

\def\bargam{\bar \gamma}
\def\barsig{\bar \sigma}
\def\barf{\bar f}
\def\baru{\bar u}
\def\barueps{{\bar u}_\eps}
\def\barfeps{\barf_\eps}

\def\div{\text{\rm div}}
\def\vn{\v \cdot \n}

\def\C{C}
\def\Ceps{{\C_\eps}}
\def\Cepsm{{\C_\eps^{-1}}}

\begin{document}
\title[Diffusion asymptotics with low regularity]{Diffusion asymptotics for linear transport with low regularity}
\author[H. Egger]{Herbert Egger$^\dag$}
\author[M. Schlottbom]{Matthias Schlottbom$^\dag$}
\thanks{$^\dag$Numerical Analysis and Scientific Computing, Department of Mathematics, TU Darmstadt, Dolivostr. 15, 64293 Darmstadt. \\
Email: {\tt $\{$egger,schlottbom$\}$@mathematik.tu-darmstadt.de}}

\begin{abstract}
We provide an asymptotic analysis of linear transport problems in the diffusion limit under minimal 
regularity assumptions on the domain, the coefficients, and the data. The weak form of the limit equation 
is derived and the convergence of the solution in the $L^2$ norm is established 
without artificial regularity requirements. This is important to be able to deal with 
problems involving realistic geometries and heterogeneous media. 
In a second step we prove the usual $O(\eps)$ convergence rates under very mild additional assumptions. 
The generalization of the results to convergence in $L^p$ with $p \ne 2$ and some limitations
are discussed. 
\end{abstract}

\maketitle



\section{Introduction}

\noindent
We study stationary mono-kinetic linear transport problems
\begin{align} 
 \vd \ueps + \gameps \ueps &= \sigeps (K-I) \ueps + \feps \qquad \text{in } \R \times \V  \label{eq:rte1},\\
  \ueps &= \geps \qquad \text{on } \dR \times \V, \ \n \cdot \v < 0 \label{eq:rte2}.
\end{align}
Here $\R$ is some bounded domain and $\V$ is the sphere with unit surface area.
The density $\ueps$ depends on position and propagation velocity. 
We are interested in the diffusive regime with coefficients and data of the form
\begin{align} \label{eq:ass1}
\gameps= \eps \bargam(\x), \ 
\sigeps = \frac{1}{\eps} \barsig(\x), \quad 
\feps = \eps \barf(\x), \quad \text{and} \quad
\geps = 0.
\end{align}
The parameter $\eps$ is small and has the physical meaning of a mean free path. 
In the case of isotropic scattering, one has 
\begin{align} \label{eq:ass2}
(Ku)(\x,\v) = \int_{\V} u(\x,\v') \d\v' =: \baru(\x).
\end{align}
The bar symbol is used for the velocity average but also to denote functions that do not depend on $\v$. 
More general assumptions than \eqref{eq:ass1}--\eqref{eq:ass2} will be considered below.
Following \cite{HabMat75,LarsenKeller74}, the solution $\ueps$ of the transport problem 
\eqref{eq:rte1}--\eqref{eq:rte2} can be formally written as
\begin{align} \label{eq:exp}
 \ueps = u_0 + \eps u_1 + \eps^2 u_2 + \ldots
\end{align}
Substituting this expansion into the governing equations
and balancing terms with the same power of $\eps$ allows to show by formal arguments that 
the lowest order term is independent of the velocity, i.e., $u_0 = \baru_0$, and 
corresponds to the unique solution of the diffusion problem
\begin{align}
-\div \big( \tfrac{1}{3\barsig} \nabla \baru_0 \big) + \bargam \baru_0 &= \barf \qquad \text{in } \R, \label{eq:da1}\\
 \baru_0 &= 0 \qquad \text{on } \dR. \label{eq:da2}
\end{align}
%
%
Note that when using the expansion \eqref{eq:exp} in the derivation of the limiting diffusion equation, 
one implicitly assumes that 
\begin{align} \label{eq:approximation}
 \ueps = \baru_0 + O(\eps), \qquad \eps \to 0,
\end{align}
which means that $\bar u_0$ is the limit of $\ueps$ as $\eps \to 0$ and that the error in the approximation is of order $\eps$.
The validity and the precise meaning of the formula \eqref{eq:approximation} is, however, not clear without further reasoning.
As we will indicate in our discussion, it is not even true, in general, without further assumptions.
The formal asymptotic argument can be made rigorous by carefully estimating the remainder in the expansion. 
This has been elaborated, e.g., in \cite{BardosSantosSentis84,BlankenshipPapanicolaou78,DL93vol6},  under rather strong additional regularity assumptions on the domain, the coefficients, and the data. The validity of the asymptotic results without such strong conditions seems to be not settled completely; see however \cite{BardosBernardGolseSentis2014} for the time-dependent case and \cite{BardosGolsePerthameSentis1988} for a related nonlinear problem.

In this paper, we address this question: if, and in what sense, 
does $\ueps$ approach the diffusion limit $\baru_0$, under minimal regularity assumptions. 
We feel that this issue is important in practice, where 
heterogeneous materials,  discontinuous coefficients, and irregular geometries 
may arise. In all these cases, the results of \cite{BardosSantosSentis84,DL93vol6} are not applicable directly
and the use of the diffusion approximation lacks a rigorous justification.
Our main result, which partially closes this gap, can be summarized as follows:
\begin{theorem} \label{thm:main}
Assume that \eqref{eq:ass1}--\eqref{eq:ass2} holds 
and let $\ueps$ and $\baru_0$ be the solutions of \eqref{eq:rte1}--\eqref{eq:rte2} and \eqref{eq:da1}--\eqref{eq:da2}, respectively. Then the following assertions hold:\\
(i) 
If $0 < c \le \bargam(\x), \barsig(\x) \le c^{-1}$ for some $c>0$, and $\barf \in L^2(\R)$, then 
\begin{align} \label{eq:res1}
\|\ueps - \baru_0\|_{L^2(\D)} = o(1), \qquad \eps \to 0.
\end{align}
(ii)
If, in addition,  $\dR \in C^{1,1}$, $\barf \in L^p(\R)$, $p>3$, and $\barsig  \in  W^{1,\infty}(\R)$, then
\begin{align} \label{eq:res2}
\|\ueps - \baru_0\|_{L^2(\D)} = O(\eps), \qquad \eps \to 0.
\end{align}
The constants of the latter estimate only depend on the bound $c$ for the coefficients, the function $\barf$, and on the domain $\R$.
\end{theorem}
We will prove this result under more general assumptions than \eqref{eq:ass1}--\eqref{eq:ass2} on the scattering operator and the parameters; see Theorems~\ref{thm:main1} and \ref{thm:main2} below.
Before we proceed, let us put our results into perspective and shortly describe the main tools used to derive them:  
Due to its practical relevance, the asymptotic analysis of the diffusion limit has stimulated a vast amount of literature:
The formal arguments of the seminal papers \cite{HabMat75,LarsenKeller74} have been made rigorous in \cite{BardosBernardGolseSentis2014,BardosGolseSentisPerthame1987,BardosSantosSentis84} and in \cite{BardosGolsePerthameSentis1988} for a nonlinear transport equation; see also \cite{DL93vol6} for a self-contained presentation of the main results.
More general problems have been considered, e.g., in \cite{DegondMas-Gallic87,GoudonMellet2003,GuoHan2012}. 
Similar formal arguments are also used for the construction of asymptotically stable numerical schemes, see e.g.   \cite{Germogenova98,Larsen2010,LarsenMorel89,LarsenMorelMiller87}.
We see our results as an extension to those presented in \cite{DL93vol6} and refer to this monograph also for further results and many more references. 

Under the assumptions of Theorem~\ref{thm:main}, we will show by variational arguments that the solutions $\ueps$ of \eqref{eq:rte1}--\eqref{eq:rte2} together with their directional derivatives are uniformly bounded in $L^2(\D)$ for all $\eps>0$. 
Similar estimates have also been derived in \cite{BardosGolsePerthameSentis1988} for a nonlinear variant of \eqref{eq:rte1}.

These a-priori estimates allow us to extract a weakly convergent subsequence which is shown to converge strongly in $L^2$ to a limit 
which is independent of the velocity variable. 
We will verify that any such limit satisfies the weak form of the diffusion problem \eqref{eq:da1}--\eqref{eq:da2} which yields the first statement of the theorem. Our derivation thus provides a rigorous justification of the diffusion approximation without regularity requirements for the coefficients or the domain.
Let us mention, that our arguments are similar to those of \cite{BardosGolsePerthameSentis1988}, where however the special nonlinear structure $\barsig=\barsig(\baru_\eps)$ was exploited in a crucial way. The results of \cite{BardosGolsePerthameSentis1988} thus do not carry over immediately to the case considered here.

When completing this manuscript we became aware of the very interesting work \cite{BardosBernardGolseSentis2014}, where the diffusion approximation for time-dependent linear transport is analyzed. 
Employing energy estimates, the authors of \cite{BardosBernardGolseSentis2014} obtain strong $L^2$-convergence of the solution to the time-dependent transport equation to a solution of a parabolic equation under rather general assumptions on the parameters and the domain $\R$.
Quantitative estimates for the convergence like \eqref{eq:res2} are not investigated in \cite{BardosBernardGolseSentis2014}.
Our results concern stationary transport and thus complement those of \cite{BardosBernardGolseSentis2014}.


%

For the derivation of the quantitative estimate \eqref{eq:res2} of Theorem~\ref{thm:main}, 
we employ an expansion of the form $\ueps = \baru_0 + \eps u_1 + \psi_\eps$. 
While we use the usual choice $u_1 = -\frac{1}{3\barsig} \vd \baru_0$ for the first order term, 
we do not use a second order term $\eps^2 u_2$, as e.g. in \cite{BardosSantosSentis84,DL93vol6}, here. 
This allows us to substantially relax the regularity assumptions on the coefficients and the data required for the asymptotic analysis. 
To establish the $O(\eps)$ bound for the remainder $\psi_\eps$, we 
utilize sharp a-priori bounds for the solutions of \eqref{eq:rte1}--\eqref{eq:rte2} in $L^p$ spaces 
which are obtained again by variational arguments. 

%
Let us remark at this point that the results of Theorem~\ref{thm:main} can be extended to $L^p$ spaces with $p \ne 2$ to some extent and without further regularity assumptions. This will be discussed at the end of the manuscript. Of course, the asymptotic estimates of \cite{BardosSantosSentis84,DL93vol6} can be applied directly for sufficiently regular problems.
%

\bigskip

The outline of the paper is as follows:
In the next section we will introduce the basic notation used throughout the manuscript. 
In Section~\ref{sec:prelim} we fix the main assumptions on the parameters and the scattering operator and state some preliminary results. 
Sharp a-priori estimates on the solution to the transport equation are then derived in Section~\ref{sec:apriori}. 
Using these a-priori estimates convergence of the solutions to the transport equation as $\eps$ tends to zero is then shown in Section~\ref{sec:convergence}. 
The quantitative estimates are established in Section~\ref{sec:rates} under some mild additional regularity assumptions.
We conclude with a short discussion of our results, about their limitation, and possible generalizations.

\section{Notation} \label{sec:notation}

\noindent
For the rest of the manuscript we assume that  
\begin{itemize} \parskip1ex 
 \item[(A1)] $\V=\{v \in \RR^3 : 4 \pi |v|^2 = 1\}$ is the sphere with unit surface measure and $\R \subset \RR^3$ is a bounded Lipschitz domain.
\end{itemize}
%
The Lebesgue spaces of measurable functions on $\D$ are denoted by $L^p(\D)$, $1 \le p \le \infty$.
These spaces are complete with respect to the norm  
$\|u\|_{L^p(\D)} = \big( \int_\D |u|^p \d(\x,\v) \big)^{1/p}$
and $L^2(\D)$ is a Hilbert space with scalar product 
$
(u,w) = (u,w)_{L^2(\D)} = \int_\D u w \d(\x,\v).
$
We will use similar notation for integration and spaces defined over other domains. 

Since the boundary $\dR$ is regular, we can define for a.e. $\x \in \dR$ the outer unit normal vector $\n(\x)$ and 
we denote by 
\begin{align*}
\Gamma_\pm := \{(\x,\v) \in \dD : \pm \n(\x) \cdot \v > 0\}
\end{align*}
the inflow and outflow part of the boundary, respectively. 
We write
\begin{align} \label{eq:tsp}
L^p(\dD;|\vn|) = \overline{L^p(\dD)}^{\|\cdot\|_{L^p(\dD;|\vn|)}},  
\end{align}
the weighted $L^p$ space of functions defined over the boundary 
with finite norm
$\|u\|_{L^p(\dD;|\vn|)} = \big( \int_\dD |u|^p |\vn| \d(\x,\v)\big)^{1/p}.$
These spaces are complete by construction and the choice $p=2$ yields again a Hilbert space. 
Spaces of functions over $\Gamma_\pm$ can be defined by restriction.

By $\nabla u(\x,\v)$ we denote the partial gradient of a smooth function with respect to the spatial variables $\x$, and we call $\vd u(\x,\v)$ the directional derivative. 
According to \cite{Agoshkov98,Cessenat84}, any function in $L^p(\D)$ having weak directional derivatives in $L^p(\D)$ 
has a well-defined trace on $\dD$. This allows us to define the energy spaces
\begin{align*}  
 \VV^p  = \{ u \in L^p(\D) : & \ \vd u \in L^p(\D),  u|_{\dD} \in L^p(\dD;|\vn|)\}
\end{align*}
with norm $\|u\|_{\VV^p} = (\|u\|_{L^p(\D)}^p +  \|\vd u\|_{L^p(\D)}^p + \|u\|_{L^p(\dD;|\vn|)}^p)^{1/p}$. 
The spaces $\VV^p$ are again complete and the case $p=2$ yields a Hilbert space.
By definition, the trace operator is a continuous mapping from $\VV^p$ to $L^p(\dD;|\vn|)$. 
Via a density argument, the Green's formula 
\begin{align} \label{eq:green}
(\vd u, w)_{\D} = -(u, \vd w)_{\D} + (\vn \, u, w)_{\dD} 
\end{align}
can be shown to hold for all functions $u \in \VV^p$, $w \in \VV^{p'}$ with $\frac{1}{p}+\frac{1}{p'}=1$. 
For  $u \in L^p(\D)$, we define the velocity average by
\begin{align} \label{eq:avg}
 \baru(\x) =  \int_\V u(\x,\v) \d\v = (u(\x,\cdot), 1_\v)_\V.
\end{align}
Here  $1_\v$ denotes the constant function in $L^\infty(\V)$.
By Fubini's theorem, $\baru$ is a function in $L^p(\R)$. 
We will tacitly identify functions in $L^p(\R)$ with the corresponding functions in $L^p(\D)$ which are independent of $\v$. 
Any function $u \in L^2(\D)$ can then be decomposed into 
\begin{align} \label{eq:split}
 u = \bar u + (u-\bar u),
\end{align}
and this splitting is orthogonal in $L^2(\D)$. 
One can show that 
\begin{align} \label{eq:avglem}
\text{the mapping } \pi : \VV^p \to L^p(\R), \quad u \mapsto \bar u \quad \text{is compact.}
\end{align}
Such statements are known as averaging lemmas and will be used in our analysis below. 
Let us refer to \cite{deVorePetrova00,GerardGolse1992,GolseLionsPerthameSentis88} for details and generalizations.

To simplify notation, we will sometimes write $a \preceq b$ or $a = O(b)$ meaning $a \le C b$ with a constant depending only on the domain or other generic constants. We also use $a \approx b$ to abbreviate $a \preceq b$ and $b \preceq a$. The inner product of two vectors is denoted by $a \cdot b$ or $a^\top b$, and we write $a \otimes b$ for the matrix $a b^\top$.

\section{Assumptions and preliminaries} \label{sec:prelim}

\noindent
To ensure the well-posedness of the radiative transfer problems 
\eqref{eq:rte1}--\eqref{eq:rte2} and in order to derive uniform a-priori bounds, 
we assume that  
\begin{itemize} \parskip1ex
 \item[(A2)] $\gameps(\x,\v) = \eps \bargam(\x)$ 
             with $c \le \bargam(\x) \le c^{-1}$ for a.e. $\x \in \R$,
 \item[(A3)] $\sigeps(\x,\v) = \frac{1}{\eps} \barsig(\x)$
             with $c \le \barsig(\x) \le c^{-1}$ for a.e. $\x \in \R$,
\end{itemize}
with some constant $c>0$. 
In addition, we impose the following abstract assumptions on 
the scattering operator 
\begin{itemize}\parskip1ex
 \item[(A4a)] $K$ is a self-adjoint positive linear operator on $L^2(\V)$.
 \item[(A4b)] For all $1 \le p \le \infty$ there holds $\|Ku\|_{L^p(\V)} \le \|u\|_{L^p(\V)}$.
 \item[(A4c)] $N(I-K)=\{u \in L^2(\V) : u(\v) = \baru\}$.
 \item[(A4d)] $(I-K) u = f$ is solvable in $L^2(\V)$, if, and only if, $\barf=0$. 
              In this case, any solution has the form $u(\v)=u_0 + u_1(\v)$ where
              $u_1$ has zero average and $\|u_1\|_{L^2(\V)} \le c_K \|f\|_{L^2(\V)}$ for some $c_K\geq 1$.
\end{itemize}
Note that the assumptions (A4) are valid for the isotropic scattering operator \eqref{eq:ass2}, but also for more general scattering operators of the form 
$$ 
  (K u)(\v) = \int_\V k(\v,\v') u(\v') \d\v' 
$$
under mild and physically reasonable conditions on the scattering kernel, e.g., 
when $k$ is positive, bounded, and symmetric with respect to $\v$ and $\v'$; see for instance \cite{DL93vol6}.
As a consequence of the assumptions (A4), the eigenvalues of the operator $I-K$ lie in the the interval $[0,1]$ 
and $0$ is a simple eigenvalue with eigenspace $N(I-K) = {\rm span}\{1_\V\}$.
Due to (A4d), one has
\begin{align}\label{eq:ev}
   c_K^{-1} \|u\|_{L^2(\V)}^2 \leq  ((I-K)u,u)_{L^2(\V)} \leq \|u\|_{L^2(\V)}^2\quad\text{for } u\in N(I-K)^\perp.
\end{align}
This provides upper and lower norm estimates for the Moore-Penrose inverse $(I-K)^\dag$ as operator on $L^2(\D)$.  
%
%
The scattering operator extends to an operator on $L^p(\D)$ by setting 
$(Ku)(\x,\v)= (K u(\x,\cdot))(\v)$. 
%
%
By 
\begin{align} \label{eq:Ce}
 \Ceps : L^2(\D) \to L^2(\D), \qquad u \mapsto \gameps u + \sigeps (I-K) u,
\end{align}
we denote the collision operator which by (A2)--(A4) can be shown to be linear, self-adjoint,  bounded, and strictly positive on $L^2(\D)$.
The transport problem \eqref{eq:rte1}--\eqref{eq:rte2} can then be written in compact form as 
\begin{align} \label{eq:rte}
\vd \ueps + \Ceps \ueps = \feps \quad \text{in } \D,
\qquad  
\ueps = \geps \quad \text{on } \G_-.
\end{align}
We will denote by
\begin{align} \label{eq:Cenorm}
(u,v)_\Ceps = (\Ceps u, v) 
\quad \text{and} \quad  
\|u\|_\Ceps = (u,u)_\Ceps^{1/2},
\end{align}
the scalar product and norm induced by $\Ceps$ and use similar notation for the norm induced by the inverse $\Cepsm$. 
Using assumptions (A2)--(A4), one easily verifies that 
\begin{align} \label{eq:equiv1}
\|u\|_\Ceps^2 &\approx \frac{1}{\eps} \|u-\baru\|_{L^2(\D)}^2 + \eps \|\baru\|_{L^2(\D)}^2, 
\\ \label{eq:equiv2}
 \|u\|_{\Cepsm}^2 &\approx \eps \|u-\baru\|_{L^2(\D)}^2 + \frac{1}{\eps}\|\baru\|_{L^2(\D)}^2,
\end{align}
holds for all $u \in L^2(\D)$. 
The constants in these norm equivalence estimates only depend on the constants of the assumptions (A2)--(A4).

\section{A-priori estimates} \label{sec:apriori}

\noindent
Let us recall the following well-posedness result from \cite{ES13a}; see also \cite{DL93vol6}.
\begin{lemma} \label{lem:aprioriLp}
Assume that conditions (A1)--(A4) hold. Then for any $\eps>0$, $\feps \in L^p(\D)$, and $\geps \in L^p(\G_-;|\vn|)$, 
the linear transport problem \eqref{eq:rte1}--\eqref{eq:rte2} has a unique solution $\ueps \in \VV^p$.
In addition, there holds
\begin{align} \label{eq:aprioriLp}
&\|\ueps\|_{L^p(\D)} + \eps^{-1/p} \|\ueps\|_{L^p(\G_+;|\vn|)} + \eps \|\vd \ueps\|_{L^p(\D)} \\ 
& \qquad \qquad \qquad \le C_p ( \eps^{-1} \|\feps\|_{L^p(\D)} + \eps^{-1/p}\|g\|_{L^p(\G_-;|\vn|)} ) \notag.
\end{align}
The constant $C_p$ of this a-priori estimate depends only on $p$, the domain $\R$, and the constants appearing in the assumptions.
\end{lemma}
This result particularly holds for $p =\infty$. 
%
%
For our analysis, we will utilize the following sharper estimates which can be obtained in the $L^2$ setting.
\begin{lemma} \label{lem:aprioriL2}
Let (A1)--(A4) hold and let $\ueps$ be the solution of \eqref{eq:rte1}--\eqref{eq:rte2} with $\feps \in L^2(\D)$ and $\geps \in L^2(\G_-;|\vn|)$.
Then
\begin{align*}
 &\|\ueps\|^2_{L^2(\G_+;|\vn|)} + \frac{1}{\eps}\|\ueps - \barueps\|_{L^2(\D)}^2 + \eps \|\barueps\|_{L^2(\D)}^2 
 \\ &\qquad \qquad \qquad \le C \big( \|g_\eps\|^2_{L^2(\G_-;|\vn|)}  + \frac{1}{\eps}\|\bar f_\eps\|_{L^2(\D)}^2 + \eps \|f_\eps-\bar f_\eps\|_{L^2(\D)}^2\big)
\end{align*}
with $C$ only depending on the domain and the constants of the assumptions.
\end{lemma}
\begin{proof}
Recall that $\bar u(\x) = \int_\V u(\x,\v) \d\v'$ denotes the velocity average.
We multiply \eqref{eq:rte1} by $\ueps$ and integrate over $\D$ to get
\begin{align} \label{eq:lemap1}
 (\v \cdot \nabla \ueps, \ueps) + \|\ueps\|_{\C_\eps}^2 = (f_\eps, \ueps).
\end{align}
The integration-by-parts formula \eqref{eq:green} allows to recast the first term as
\begin{align*} 
 (\vd \ueps, \ueps) 
&= \frac{1}{2} \big(\|\ueps\|^2_{L^2(\Gamma_+;|\vn|)} - \|\ueps\|^2_{L^2(\Gamma_-;|\vn|)}\big).
\end{align*}
Using the boundary condition \eqref{eq:rte2}, we can substitute $\geps$ for $\ueps$ in the last term.
The right-hand side of \eqref{eq:lemap1} can be further estimated by 
\begin{align*} 
(\feps,\ueps) 
 &\le \tfrac{1}{2}\|\feps\|_{\Cepsm}^2 + \tfrac{1}{2} \|\ueps\|_{\Ceps}^2.
\end{align*} 
Inserting these two formulas into \eqref{eq:lemap1} and rearranging terms, we obtain
\begin{align*}
\|\ueps\|_{L^2(\G_+;|\vn|)}^2 + \|\ueps\|_{\Ceps}^2 
&\le \|\feps\|_{\Cepsm}^2 + \|\geps\|^2_{L^2(\G_-;|\vn|)}.
\end{align*}
The result now follows via the norm equivalences  \eqref{eq:equiv1} and \eqref{eq:equiv2}.
\end{proof}
Since the estimate of Lemma~\ref{lem:aprioriL2} is sharp, 
some additional conditions on the data $\feps$ and $\geps$ 
will be  required 
to obtain uniform bounds for the solution as $\eps \to 0$. 
We therefore assume in the following that
\begin{itemize} \parskip1ex
 \item[(A5)] $\|\feps\|_{L^2(\D)} = O(1)$ \ and \ $\|\barfeps\|_{L^2(\D)} = O(\eps)$. 
 \item[(A6)] $\|\geps\|_{L^2(\G_-;|\vn|)}= O(\eps^{1/2})$.
\end{itemize}
As a consequence of these assumptions and the previous lemma, we obtain
\begin{lemma}\label{lem:apriori}
Let (A1)--(A6) hold. Then
\begin{align*}
{\rm (i)} & \quad \|\ueps\|_{L^2(\G_+;|\vn|)}  = O(\eps^{1/2}), 
&
{\rm (ii)} & \quad \|\ueps - \barueps\|_{L^2(\D)} = O(\eps), \\
{\rm (iii)} & \quad \|\barueps\|_{L^2(\D)} = O(1), 
&
{\rm (iv)} & \quad \|\vd \ueps\|_{L^2(\D)} = O(1).
\end{align*}
\end{lemma}
\begin{proof}
The first three estimates follow directly from Lemma~\ref{lem:aprioriL2} and assumptions (A5)--(A6). 
Multiplying equation \eqref{eq:rte1} by $\Cepsm \vd \ueps$ and integrating over $\D$ yields
\begin{align} \label{eq:leml2}
\|\vd \ueps\|_{\Cepsm}^2 + (\ueps, \vd \ueps) &= (\feps, \vd \ueps)_\Cepsm. 
\end{align}
Similar as in the proof of Lemma~\ref{lem:aprioriL2}, the second term in \eqref{eq:leml2} can be 
replaced by boundary terms, and the third term  can be estimated by
\begin{align*}
 (\feps, \vd \ueps)_{\Cepsm} \le \frac{1}{2}\|\feps\|_{\Cepsm}^2 + \frac{1}{2} \|\vd \ueps\|^2_{\Cepsm}.
\end{align*}
Substituting these expressions into \eqref{eq:leml2} yields
\begin{align*}
\|\vd \ueps\|_{\Cepsm}^2 + \|\ueps\|_{L^2(\G_+;|\vn|)}^2 
 &\le \|\feps\|_{\Cepsm}^2 + \|\geps\|_{L^2(\G_-;|\vn|)}^2. 
\end{align*}
Assertion (iv) now follows via the norm equivalences \eqref{eq:equiv1} and \eqref{eq:equiv2}. 
\end{proof}

Let us mention, that similar estimates as those of Lemma~\ref{lem:apriori} have been obtained in \cite{BardosGolsePerthameSentis1988} for a nonlinear transport equation.

\section{Convergence to the diffusion approximation} \label{sec:convergence}

\noindent
We now consider a sequence of problems \eqref{eq:rte1}--\eqref{eq:rte2} with coefficients and data satisfying (A2)--(A3) and (A5)--(A6) for parameter $\eps$ tending to zero. 
%
%
%

\begin{lemma} \label{lem:weak} 
Let (A1)--(A6) hold and assume that $\frac{1}{\eps}\barfeps \to \barf \in L^2(\R)$ and $\feps-\barfeps\to 0$ when $\eps \to 0$. 
Then $\ueps \rightharpoonup u_0$ weakly in $\VV^2$ and $\ueps \to u_0$ strongly in $L^2(\D)$ for some $u_0 \in L^2(\D)$. 
Moreover, the limit is independent of the velocity variable $\v$, i.e., $u_0 = \bar u_0$.
\end{lemma}
\begin{proof}
By Lemma~\ref{lem:apriori}, the solution $\ueps$ is uniformly bounded in $\VV^2$ as $\eps \to 0$. 
This allows us to extract a weakly convergent subsequence, again denoted by $\ueps$, such that $\ueps \rightharpoonup u_0$ in $\VV^2$.  By the averaging lemma \eqref{eq:avg}, 
we know that $\baru_\eps \to \baru_0$ strongly in $L^2$, and by estimate (ii) of Lemma~\ref{lem:apriori} we conclude that $\ueps \to \bar u_0$ strongly in $L^2$. 
Since the limit is unique, we deduce that $u_0 = \baru_0$, which shows the assertions 
for a subsequence. 
Convergence of the whole sequence will be obtained below by showing that the limit is independent of the subsequence.
\end{proof}
Next, we show that any limit of Lemma~\ref{lem:weak} satisfies a diffusion problem similar to \eqref{eq:da1}--\eqref{eq:da2}. 
Since the limit of any subsequence satisfies the same equation, we obtain the convergence of the full sequence
which completes the proof of the previous lemma.
We use the symbol $\baru_0$ for the limit to emphasize that it is independent of $\v$ and proceed as follows:

\smallskip

\noindent 
{\em Step 1:} 
Since $\baru_0$ is independent of $\v$ and $\int_\V \v_i \v_j \d\v = \frac{1}{3} \delta_{ij}$, we get
\begin{align*}
 \|\vd \bar u_0 \|_{L^2(\D)}^2 
 &= \int_\R \nabla \baru_0 \cdot \int_\V \v \otimes \v \d\v \nabla \baru_0 \d\x 
  = \frac{1}{3} \|\nabla \baru_0\|_{L^2(\R)}^2.
\end{align*}
Hence $\baru_0 \in H^1(\R)$, and from Lemma \ref{lem:apriori} (i), condition (A6), and continuity of the trace operator, we infer that $\bar u_0 = 0$ on $\dR$; thus $\baru_0 \in H_0^1(\R)$.

\smallskip 

\noindent 
{\em Step 2:}
Testing equation \eqref{eq:rte1} with $\bar \psi \in H_0^1(\R)$ yields
\begin{align} \label{eq:testpsi}
 (\barfeps,\bar \psi) = (\vd \ueps, \bar\psi) + \eps (\bargam \ueps, \bar\psi).
\end{align}

\smallskip

\noindent 
{\em Step 3:}
Let $\bar\psi$ be as in Step~2. 
Using (A4), we can define a test function 
$$
  \phipsi = \frac{1}{\barsig}(I-K)^\dag \vd  \bar\psi
$$
with zero velocity average. Since the Moore-Penrose inverse $(I-K)^\dag$ is a linear bounded operator, see \eqref{eq:ev},
we conclude that
$$
\|\phipsi\|_{L^2(\R\times\V)} \preceq \|\vd \bar \psi\|_{L^2(\D)} \preceq \|\nabla \bar\psi\|_{L^2(\R)}.
$$

\smallskip

\noindent 
{\em Step 4:}
Testing \eqref{eq:rte1} with the function $\phipsi$ as defined in Step~3, we obtain
\begin{align} \label{eq:testphi1}
  (\feps,\phipsi)
  = (\vd \ueps, \phipsi) + (\Ceps \ueps, \phipsi). 
\end{align}
Using the definition of $\Ceps$ and of $\phipsi$, we can express the third term as
\begin{align*} 
 (\Ceps \ueps, \phipsi)
 &= \eps (\bargam \ueps, \phipsi) + \frac{1}{\eps} (\ueps, \vd \bar\psi)
\end{align*}
and for the last term of this expression, we employ the Green's formula \eqref{eq:green} and the boundary conditions for $\bar\psi$, to obtain
\begin{align*} 
 (\ueps, \vd \bar\psi) = -(\vd \ueps,\bar\psi).
\end{align*}
Substituting these two expressions into \eqref{eq:testphi1} yields 
\begin{align} \label{eq:testphi}
( \feps,\phipsi) =  (\vd \ueps, \phipsi) + \eps (\bargam \ueps, \phipsi) - \frac{1}{\eps}  (\vd \ueps, \bar\psi).
\end{align}

\smallskip

\noindent 
{\em Step 5:}
By adding $\frac{1}{\eps}$ times \eqref{eq:testpsi} and \eqref{eq:testphi}, we now see that
the solutions $\ueps$ of the transport problem \eqref{eq:rte1}--\eqref{eq:rte2} satisfy the variational principle 
\begin{align} \label{eq:step5}
(\feps-\barfeps,\phipsi) + \frac{1}{\eps}
(\barfeps,\bar\psi) 
 &= (\bargam \ueps, \bar\psi) + (\vd \ueps, \phipsi) + \eps (\bargam \ueps, \phipsi)
\end{align}
for any $\bar\psi \in H_0^1(\R)$ and with $\phipsi$ defined as in Step~3.

\smallskip

\noindent
{\em Step 6:}
Since $\phipsi$ and $\ueps$ are uniformly bounded in $L^2(\D)$, the last term of \eqref{eq:step5} tends to zero as $\eps \to 0$. 
The third term of \eqref{eq:step5} can be expanded as
\begin{align} \label{eq:step6}
(\vd \ueps, \phipsi) 
&= (\vd \ueps - \vd\bar u_0, \phipsi) + (\vd\baru_0, \phipsi).
\end{align}
Since $\ueps$ converges to $\baru_0$ weakly in $\VV^2$, the second term of \eqref{eq:step6} vanishes with $\eps \to 0$,
and, using the definition of $\phipsi$, the third term of equation \eqref{eq:step6} can be written as
\begin{align*}
(\vd\baru_0, \phipsi) = \int_\R \nabla \baru_0 \cdot  \frac{1}{\barsig}\int_\V \v (I-K)^\dag \v^\top \d\v \nabla \bar\psi \d\x,
\end{align*}
where $(I-K)^\dag$ is applied to the components of the function $\v\mapsto \v^\top$.
 
\smallskip

\noindent 
{\em Step 7:}
For a.e. $\x \in \R$, we can define a $3 \times 3$ matrix 
\begin{align} \label{eq:A}
 \bar A(\x) = \frac{1}{\barsig(\x)}  \int_\V \v (I-K)^\dag \v^\top \d\v.
\end{align}
Since for  $\xi\in \RR^3$, the function $v^\top\xi$ is orthogonal to $1_\v$ in $L^2(\V)$, 
we conclude from estimate \eqref{eq:ev} that
\begin{align*}
  \xi^\top \bar A(\x) \xi = \frac{1}{\barsig(x)} ( v^\top\xi, (I-K)^\dagger v^\top\xi)_{L^2(\V)} \geq \frac{1}{\barsig(x)} \| v^\top\xi\|_{L^2(\V)}^2 = \frac{1}{3\barsig(x)} |\xi|^2.
\end{align*}
In a similar way one can show a uniform upper bound for $\bar A(x)$. 
Thus $\bar A(x)$ is symmetric, positive definite and bounded uniformly for a.e. $\x \in \R$.

\smallskip

\noindent 
{\em Step 8:}
Taking the limit $\eps \to 0$ in \eqref{eq:step5}, using the previous estimates, and assuming that $\frac{1}{\eps} \bar \feps \to \bar f$ and $\feps-\barfeps\to 0$, 
we now see that any limit $\baru_0$ of a sequence of solutions $\ueps$ with $\eps \to 0$ satisfies 
\begin{align} \label{eq:weak}
(\bar A \nabla \baru_0, \nabla \bar \psi) + (\bar \gamma \baru_0, \bar \psi) = (\bar f, \bar\psi) \qquad \text{for all } \bar\psi \in H_0^1(\R).
\end{align}
This is the weak form of a diffusion equation, and the limit $\baru_0$ is uniquely characterized by this equation. 
We therefore have proven

\begin{theorem} \label{thm:main1}
Let (A1)--(A6) hold and assume that $\frac{1}{\eps} \barfeps \to \barf \in L^2(\R)$ and $\feps-\barfeps\to 0$ in $L^2(\D)$.
Then the solutions $\ueps$ of \eqref{eq:rte1}--\eqref{eq:rte2} converge weakly in $\VV^2$ and strongly in $L^2(\D)$ to the weak solution $\baru_0 \in H_0^1(\R)$ of the diffusion problem
\begin{align} \label{eq:diffapp}
-\div(\bar A \nabla \baru_0) + \bargam \baru_0 &= \barf \quad \text{in } \R, 
&
\baru_0 &= 0 \quad \text{on } \dR
\end{align}
with diffusion tensor $\bar A$ defined as in \eqref{eq:A}.
\end{theorem}
For the scattering operator of the form \eqref{eq:ass1}, one has $\bar A(x) = \frac{1}{3\barsig(x)} I$, and we thus obtain the first part of Theorem~\ref{thm:main} as a special case. 
Let us note that for our arguments to hold we did not require an $L^\infty$ bound on the solution nor its nonnegativity.


\section{Quantitative estimates} \label{sec:rates}

\noindent

The convergence obtained in Theorem~\ref{thm:main1} may be arbitrarily slow in general.
To obtain convergence rates, we will now further estimate the approximation error 
$\|\ueps - \baru\|_{L^2(\D)}$. Let us formally
write the solution as
\begin{align} \label{eq:expansion}
\ueps = \baru_0 + \eps u_1 + \psi_\eps,
\end{align}
where $\baru_0$ denotes the diffusion limit and $u_1 = -\frac{1}{\barsig} (I-K)^\dag \vd \baru_0$.
By the results of the previous section, 
we already know that $\psi_\eps$ converges to zero
if the assumptions of Theorem~\ref{thm:main1} are valid.
We will show in the following that the remainder $\psi_\eps$ can be bounded
in terms of $\eps$ under some mild additional regularity assumptions.

\smallskip

\noindent
{\em Step 1:}
We start by stating a regularity result for the diffusion limit $\baru_0$.
Assume that $\dR \in C^{1,1}$, $\barsig \in W^{1,\infty}(\R)$, and $\barf \in L^p(\R)$ for some $p>3$. 
Then from elliptic regularity results \cite{GT}, we know that $\baru_0 \in W^{2,p}(\R)$ with $\|\baru_0\|_{W^{2,p}(\R)} \preceq \|\barf\|_{L^p(\R)}$. 
Note that, as a consequence, the diffusion equation \eqref{eq:diffapp} holds pointwise a.e.\@ in $\R$.
From standard embedding theorems, we further deduce that $\nabla \baru_0$ is continuous and $\|\nabla \baru_0\|_{L^\infty(\dR)} \preceq \|\barf\|_{L^p(\R)}$.

\smallskip

\noindent
{\em Step 2:}
Using the definition of the remainder $\psi_\eps$ and equation \eqref{eq:rte1}, we obtain
\begin{align*}
\vd \psi_\eps + \Ceps \psi_\eps 
 &= \feps - \vd (\baru_0 + \eps u_1) - \Ceps (\baru_0 + \eps u_1) = (*). 
\end{align*}
The right-hand side of this equation can be further expanded as
\begin{align*}
(*) &= -\frac{\barsig}{\eps} (I-K)\baru_0 - (\vd \baru_0 +\barsig(I-K) u_1) 
\\  & \qquad \qquad  
+ \eps (\frac{1}{\eps} \feps - \vd u_1 -  \bargam \baru_0) - \eps^2 \bargam u_1 = \text{(i)}+\text{(ii)}+\text{(iii)}+\text{(iv)}.
\end{align*}
The term (i) vanishes, since $\baru_0$ lies in the kernel of $(I-K)$. 
The second term (ii) is zero by definition of $u_1$.
Since $\baru_0$ solves \eqref{eq:diffapp} pointwise, we get
\begin{align*}
\text{(iii)} &= (\feps - \eps \barf) + \eps (-\div(A \nabla \baru_0) - \vd u_1) = \text{(iiia)} + \text{(iiib)}.
\end{align*}
The term (iiia) can be treated by assumptions on $\feps$, e.g., by requiring that $\feps = \eps \barf$. 
From the definition of $\bar A$ and of $u_1$, we then deduce  that
\begin{align*}
\text{(iiib)} = \eps (\overline{\vd u_1} - \vd u_1) =: \eps \widetilde f
\end{align*}
where $\widetilde f$ is bounded in $L^2(\D)$ and has zero velocity average.

\smallskip 

\noindent
{\em Step 3:}
Collecting the previous formulas, we see that $\psi_\eps$ satisfies
\begin{align} \label{eq:psieps}
\vd \psi_\eps + \Ceps \psi_\eps &= \eps \widetilde f - \eps^2 \bargam u_1 &\,&  \text{in } \D, \\
\psi_\eps &= \geps - \eps u_1 &\,& \text{on } \G_-.
\end{align}
To estimate the norm of $\psi_\eps$, we split the remainder by $\psi_\eps = \psi_\eps^{f} + \psi_\eps^{g}$
where $\psi_\eps^{f}$ and $\psi_\eps^{g}$ denote the solutions of the system \eqref{eq:psieps} 
with homogeneous boundary data and right-hand side, respectively.
The two components can now be bounded independently by the a-priori estimates of Section~\ref{sec:apriori}. 

\smallskip

{\em Step 4:} 
Applying Lemma~\ref{lem:aprioriL2} with $\feps=\eps \widetilde f - \eps^2\bargam u_1$ and $\geps=0$, and noting that $\widetilde f$ and $u_1$ have 
zero velocity average, we obtain $\eps\|\psi_\eps^f\|_{L^2(\D)}^2 \preceq \eps \| \eps \widetilde f- \eps^2\bargam u_1\|^2_{L^2(\D)}$. 
Since $\widetilde f$ and $u_1$ are uniformly bounded, we conclude that $\|\psi_\eps^f\|_{L^2(\D)} = O(\eps)$.

\smallskip

{\em Step 5:}
To bound $\psi_\eps^g$, we utilize the estimate \eqref{eq:aprioriLp} for $p=\infty$ with $\feps=0$ and $\geps$ replaced by $\geps-\eps u_1$. This yields $\|\psi_\eps^g\|_{L^\infty(\D)} \preceq \|\geps - \eps u_1\|_{L^\infty(\G_-)}$. 
Assuming that $\|\geps\|_{L^\infty(\G_-)}=O(\eps)$, we can conclude that $\|\psi_\eps^g\|_{L^2(\G_-)} = O(\eps)$ by using the relation of $u_1$ and $\bar u_0$, and the estimates of Step~1. 
 
\medskip

\noindent
A combination of the previous considerations then yields

\begin{theorem} \label{thm:main2}
Let (A1)--(A6) hold, and additionally assume that
$\dR \in C^{1,1}$, $\barsig \in W^{1,\infty}(\R)$, $\feps = \eps \barf$ with $\barf \in L^p(\R)$, $p>3$, and $\|\geps\|_{L^\infty(\G_-)} = O(\eps)$. 
Then $\|\ueps - \baru\|_{L^2(\D)} = O(\eps)$, and the constant in this estimate only depends on the bounds for the coefficients and the data, and on the domain.
\end{theorem}

With the same arguments as used at the end of the previous section, 
the second assertion of Theorem~\ref{thm:main} now follows as a special case of this result.

\section{Discussion} \label{sec:discussion}

Let us mention some implications of our results and highlight questions that remain open:
Since we only consider bounded domains here, 
one directly obtains convergence and convergences rates also in $L^p$ for $1 \le p \le 2$  
under the assumptions of Theorems~\ref{thm:main1} and \ref{thm:main2}, respectively. 
In order to establish convergence in $L^p$ for $p > 2$,  we further assume that $\feps = \eps \barf$ and $\geps = \eps g$ with $\barf \in L^\infty(\R)$ and $g \in L^\infty(\G_-)$. 
Using the bounds of Lemma~\ref{lem:aprioriLp} for $p = \infty$, and the convergence in $L^2$ provided by Theorem~\ref{thm:main1}, 
one can see that $\bar u_0 \in L^\infty(\R)$, 
and that $\ueps \rightharpoonup^* u_0$ weak$^*$ in $L^\infty(\D)$ and $\ueps \to \bar u_0$ strongly in $L^p(\D)$ for all $p < \infty$; see also the results of \cite{BardosBernardGolseSentis2014} about weak$^*$ convergence for the time-dependent transport equation.
The question, if strong convergence also holds in $L^\infty$ without further regularity requirements, remains open.

With interpolation arguments, one can show that $\|\ueps - \baru_0\|_{L^p(\D)}=O(\eps^{2/p})$ under the assumptions of Theorem~\ref{thm:main2}.
It is not clear, if the full $O(\eps)$ rate holds without further assumptions. 
%
%
Under sufficient regularity, convergence in $L^p$, $1 \le p \le \infty$ and $O(\eps)$ estimates  follow from the results of \cite{BardosSantosSentis84,DL93vol6}.

For the derivation of Theorem~\ref{thm:main1}, we utilized uniform a-priori bounds for the directional derivatives $\vd \ueps$ in $L^2$. As the following argument shows, such a uniform estimate cannot hold in $L^\infty$ without spatial regularity of the parameter $\barsig$: 
Assume that $\vd \ueps$ is bounded uniformly in $L^\infty(\D)$.
Then $\bar A \nabla u_0$ must be bounded in $L^\infty(\R)$ as well, i.e., for any $\bar f \in L^\infty$ 
the solution $\bar u_0$ of the diffusion problem would lie in $W^{1,\infty}(\R)$. 
This is however not true, in general, if the coefficient $\barsig$ has jumps; see e.g. \cite{Groeger89}. 


Let us finally mention some directions in which our results can be generalized more or less directly:
One can handle more general, space dependent scattering operators $K:L^2(\D)\to L^2(\D)$, e.g., by assuming that 
the mapping $x\mapsto K(x)$ is Bochner integrable, that $K(x):L^2(\V)\to L^2(\V)$ satisfies the conditions (A4) a.e. in $\R$, and that multiplication with $\sigma$ commutes with $K$. 
Also scattering operators with more general eigenfunctions for the zero eigenvalue of $I-K$ can be considered. 
If the eigenfunction is space dependent, an additional drift term has to be included in the diffusion problem; 
see \cite{DL93vol6,GoudonMellet2003}.
The conditions on the parameters $\sigma_\eps$ and $\gamma_\eps$ can be relaxed as well, 
e.g., the parameters may depend on $\v$ in a certain form; see e.g. \cite{SanchezRagusaMasiello2008}.
As can be seen from the proofs of our results, also the conditions on the data can be relaxed to some extent. 
Let us finally mention that the extension to non-mono-kinetic problems seems possible with similar arguments as in \cite{DL93vol6} or \cite{BardosBernardGolseSentis2014}.

\bigskip

\section*{Acknowledgments}

\noindent
The first author acknowledges support by DFG (Deutsche Forschungsgemeinschaft) through grants IRTG 1529 and GSC 233.

\bibliographystyle{plain}
\bibliography{bib}

\end{document}